\documentclass[12pt]{article}
\usepackage{amsfonts}
\usepackage{amssymb}
\usepackage{amsmath}
\usepackage{amsthm}
\usepackage{graphicx}
\numberwithin{equation}{section}
\newtheorem{thm}{Theorem}[section]
\newtheorem{prop}[thm]{Proposition}

\theoremstyle{definition}
\newtheorem{exmp}{Example}
\newtheorem{rem}{Remark}
\newtheorem*{ack}{Acknowledgements}

\begin{document}

 \centerline{\LARGE  A Witt type formula  \vspace{3mm} }
 \centerline{\large G.A.T.F.da Costa
\footnote{g.costa@ufsc.br}}
\centerline{\large Departamento
de Matem\'{a}tica} \centerline{\large Universidade Federal de
Santa Catarina} \centerline{\large
88040-900-Florian\'{o}polis-SC-Brasil}

\begin{abstract}

Given a finite, connected graph G, T its edge adjacency matrix, and N, a positive integer, this paper investigates some combinatorial and algebraic properties of the
Witt type formula 
\begin{equation*}
\Omega(N,T)= \frac{1}{N}
\sum_{g|N} \mu(g) \hspace{1mm} Tr T^{\frac{N}{g}}.
\end{equation*}
The sum ranges over the positive divisors of $N$ and $\mu$ is the M\"obius function.  
 
\vspace{5mm}

\end{abstract}

\section{ Introduction} \label{sec:in}

The objective of the present paper
is to investigate algebraic and combinatorial aspects of the formula given in the abstract
which gives the number of  equivalence classes of arbitrarily oriented but non-backtracking non-periodic closed  paths of  length $N$ in a oriented graph.
The formula is well known in association with the zeta function of a graph investigated by several authors.  See \cite{te,tee,storm} and references therein. 
As shown in this paper the formula has several properties of the Witt formula type not investigated previously, as far as I know. 

Let's recall Witt formula and some of its properties. Let $N$ be a positive integer,  $R$  a real number,  $\mu$  the classical  M\"obius function defined by the rules: a)
$\mu(+1)=+1$, b) $\mu(g)=0$, if
$g=p_{1}^{e_{1}}...p_{q}^{e_{q}}$, $p_{1},...,p_{q}$ primes, and
any $e_{i}>1$, c) $\mu(p_{1}...p_{q})=(-1)^{q}$.
The polynomial of degree $N$ in $R$ with rational coefficients given in terms of M\"obius function,
\begin{equation} \label{(1)}
{\cal M}(N;R) = \frac{1}{N} \sum_{g \mid N} \mu (g) R^{\frac{N}{g}},
\end{equation}
has many applications in  algebra and combinatorics \cite{mor}. It
is called the  {\it Witt formula} when it is  
associated with the following result  \cite{ser}:
If $V$ is an $R$-dimensional vector space and $L$ is the free Lie algebra generated by  $V$ then $L$ has a 
${\mathbb{Z}}_{>0}$ 
gradation $L= \bigoplus_{N=1}^{\infty} L_{N}$, $L_{N}$ has dimension given by
${\cal M}(N;R)$ which satisfies the formal relation
\begin{equation} \label{(2)}
\prod_{N=1}^{\infty} (1-z^{N})^{{\cal M}(N;R)} = 1- Rz
\end{equation}
called the {\it Witt identity}. Notice that the  coefficient of the linear term  in the right hand side is minus the dimension of the vector space that generates the Lie algebra. Witt identity  follows from the Poincar\'e-Birkoff-Witt theorem which says that
\begin{equation} \label{(3)}
\prod_{N=1}^{\infty} (1-z^{N})^{-{\cal M}(N;R)} = 1+\sum_{N=1}^{\infty} R^{N} z^{N}
\end{equation}
 is the generating function for the dimensions of the homogeneous subspaces of the enveloping algebra of $L$. 

Witt formula  is also called the {\it necklace polynomial}  because ${\cal M}(N;R)$ gives the number of
 inequivalent non-periodic colorings of  a circular string of $N$ beads -  a necklace -  with at most $R$ colors  \cite{metro}.
 In \cite{sherm} S. Sherman 
 associated it to
the number of equivalence classes of closed
non-periodic paths of length $N$ which traverse counterclockwisely without backtracking the edges of a graph
with $R$  loops counterclockwisely oriented and hooked to a single vertex so that $\Omega$ generalizes ${\cal M}$. Notice that the coefficient of the linear term in the right hand side of Witt identity  is minus the number of loops in the graph.

 It is proved that the formula for $\Omega(N,T)$ satisfies some identities analogous to those satisfied by ${\cal M}(N,R)$ which Carlitz proved in \cite{car} and
 Metropolis and Rota  in \cite{metro}. In \cite{mor}, Moree proved similar identities for his Witt transform. Also,
the formula can be interpreted as a dimension formula
and it can be associated to a coloring problem.

The paper is organized as follows. In section 2, some preliminary definitions and results are given. In section 3, 
several identities satisfied by the formula is proved. In section 4, the formula is interpreted as a dimension formula associated to free Lie super algebras and, in section 5, to necklace colorings.  The formula is applied to some examples. 

\section{ Preliminaries} \label{sec:pre}

Let $G=(V,E)$ be a finite  connected and oriented graph where $V$ is the set of
vertices with $|V|$ elements and $E$ is the set of oriented edges with $|E|$ elements  labeled $e_{1}$,
...,$e_{|E|}$. An edge has an origin and an end  as given by its orientation. The
graph may have multiple edges and loops.

Now, consider the graph $G^{*}$ built from $G$ by adding
in the opposing oriented edges $e_{|E|+1}=(e_{1})^{-1}$,
...,$e_{2 |E|}=(e_{|E|})^{-1}$, $(e_{i})^{-1}$ being the oriented edge
opposite to $e_{i}$ and with origin (end) the end (origin) of
$e_{i}$. In the case that $e_{i}$ is an oriented loop,
$e_{i+|E|}=(e_{i})^{-1}$ is just an additional oriented loop hooked
to the same vertex. Thus, $G^{*}$  has $2|E|$ oriented
edges.

A path in $G$ is given by an ordered sequence $(e_{i_{1}},...,e_{i_{N}})$, $i_{k} \in \{1, ..., 2|E|\}$, of oriented edges  in  $G^{*}$ such that the end of $e_{i_{k}}$ is the origin of $e_{i_{k+1}}$. Also, a path can be represented by a word in the alphabet of the symbols in the set $\{e_{1}, ..., e_{2E} \}$, a word being a concatenated product of symbols which respect the order of the symbols in the sequence.

In this paper, all paths are cycles. These are non-backtracking tail-less closed paths, that is, the end of $e_{i_{N}}$ coincides with the origin of $e_{i_{1}}$, subjected to the non-backtracking condition that $e_{i_{k+1}}  \neq e_{i_{k}+|E|}$. In another words,
a cycle  never goes immediately backwards over a  previous edge. Tail-less means that $e_{i_{1}}  \neq e_{i_{N}}^{-1}$.
The length of a cycle is the number of edges in its sequence. A cycle $p$ is called periodic if $p=q^r$ for some $r>1$ and $q$ is a non periodic cycle. Number $r$ is called the period of $p$. The cycle $(e_{i_{N}}, e_{i_{1}}, ...,e_{i_{N-1}})$ is called a circular permutation of $(e_{i_{1}},...,e_{i_{N}})$
and $(e_{i_{N}}^{-1},...,e_{i_{1}}^{-1})$ is an inversion of the latter. 
A cycle and its inverse are taken as distinct.

The classical M\"obius inversion formula is used several times in this paper. Given arithmetic functions $f$ and $g$ it states that $g(n)= \sum_{d|n} f(d)$ if and only if $f(n)=\sum_{d|n} \mu(d) g(n/d)$.

In order to count cycles of a given length in a graph $G$ a crucial tool is the edge adjacency matrix of $G$ \cite{te}. This 
is the $2|E| \times 2|E|$ matrix $T$ defined as follows:
$T_{ij}=1$, if end vertex of edge $i$ is the start vertex of
edge $j$ and edge $j$ is not the inverse edge of $i$;
$T_{ij}=0$, otherwise. 

\begin{thm}\label{thm1} (\cite{te})
The number $Tr T^{N}$ (over)counts cycles 
of length $N$ in a graph $G$.
\end{thm}

\begin{proof}
Let $a$ and $b$ be two edges of $G$. The $(a,b)^{th}$ entry of
matrix $T^{N}$ is
\begin{displaymath}
(T^{N})_{(a,b)}= \sum_{e_{i_{1}}, ..., e_{i_{N-1}}} T_{(a,
e_{i_{1}})} T_{(e_{i_{1}},e_{i_{2}})}...T_{(e_{i_{N-1}},b)}
\end{displaymath}
From the definition of the entries of $T$ it follows that
$(T^{N})_{(a,b)}$ counts the number of paths of length $N$ with no backtracks from
edge $a$ to  edge $b$. For $b=a$, only closed paths are
counted. Taking the trace gives the number of non-backtracking closed paths with
every edge taken into account as starting edge, hence, the trace overcounts closed paths because every edge in the path is taken into account as starting edge.
The paths counted by the trace are tail-less, that is, $e_{i_{1}}  \neq e_{i_{N}}^{-1}$; otherwise,  $Tr T^{N}= \sum_{a} (T^{N})_{(a,a)}$ would have a term with entry $(a,a^{-1})$ which is not possible.
\end{proof}

\begin{thm} \label{thm2} \cite{tee}
 Denote by
 $\Omega(N,T)$ the number of  equivalence classes of non periodic cycles  of 
 length $N$ in $G$. This number is 
given by the following formula:
\begin{equation} \label{(4)}
\Omega(N,T)= \frac{1}{N}
 \sum_{g|N} \mu(g) \hspace{1mm} Tr T^{\frac{N}{g}}
\end{equation}
\end{thm}

\begin{proof}
In the set of $Tr T^{N}$  cycles  there
is the subset with  $N \Omega(N,T)$ elements formed by the non periodic cycles of length $N$ plus their circular
permutations 
and the subset with  $\sum_{g \neq 1|N}\frac{N}{g} \Omega(\frac{N}{g},T)$ elements formed by the periodic cycles
 of length $N$
(whose periods are the common divisors of $N$) plus their circular
permutations. 
(A cycle of period $g$ and length $N$  is of
the form
\begin{equation*}
(e_{k_{1}} e_{k_{2}}
...e_{k_{\alpha}})^{g}
\end{equation*}
where $\alpha= N/g$, and
$ (e_{k_{1}} e_{k_{2}}
...e_{k_{\alpha}}) $ is a non periodic cycle
so that the number of periodic cycles with period $g$ plus their
circular permutations  is given by $(N/g)\Omega(N/g, T)$).
Hence,
\begin{equation*} 
Tr T^{N}= \sum_{g \mid N} \frac{N}{g} \hspace{1mm} \Omega \left( \frac{N}{g},T
\right)
\end{equation*}
M\"obius inversion formula gives the result.
\end{proof}

\begin{rem} \label{rmk1}
Some terms in the right hand side of \eqref{(4)} are negative. 
In spite of that  the right hand side is always positive. Multiply both sides by $N$. The fist term equals $Tr T^{N}$
while the other terms give (in absolute value) the number according to period of the various subsets of  periodic cycles which are proper subsets of the larger set with $Tr T^{N}$ elements.
\end{rem}

\begin{rem} \label{rmk2}
Witt formula can be expressed in a form analogous to \eqref{(4)}.  Define $Q$ as the $R \times R$ matrix with all entries equal to one. The trace $Tr Q^{N}=R^{N}$  counts  counterclockwisely oriented cycles in the graph with $R$  loops counterclockwisely oriented and hooked to a single vertex so that
\begin{equation} \label{(5)}
{\cal M}(N;R) 
 = \frac{1}{N} \sum_{g \mid N} \mu (g) Tr Q^{\frac{N}{g}}
\end{equation}
\end{rem}

A recurrence relation which may be useful in practical calculations of $\Omega(N,T)$ is given next.

\begin{thm} \label{thm3}
\begin{equation} \label{(6)}
N\Omega (N,T) =   Tr T^{N} -    \sum_{g \mid N, g \neq N} g \hspace{1mm} \Omega \left( g,T
 \right)
\end{equation}
\end{thm}

\begin{proof}
This follows from
\begin{equation*}
Tr T^{N}= \sum_{g \mid N} g \hspace{1mm} \Omega \left( g,T
\right) = N\Omega (N,T) +\sum_{g \mid N, g \neq N} g \hspace{1mm} \Omega \left( g,T
 \right).
\end{equation*}
\end{proof}

\section{Some identities satisfied by $\Omega$} \label{sec:count}

It turns out that $\Omega(N,T)$ satisfies some identities analogous to those satisfied by Witt formula proved in \cite{car,metro}. These identities are established in this section. In \cite{mor}, Moree proved similar identities for his Witt transform.

\begin{thm} \label{thm4}
Given the matrices $T_{1}$ and $T_{2}$ define
$S (s,T_{i})=
 \sum_{g|s} \mu(g) \hspace{1mm} Tr T_{i}^{\frac{N}{g}}$, $i=1,2$,
and 
denote by $T_{1} \otimes T_{2}$ the Kronecker product of $T_{1}$ and $T_{2}$. Then,
\begin{equation} \label{(7)}
\sum_{[s,t]=N} S (s,T_{1}) S (t,T_{2})
= S (N,T_{1} \otimes T_{2})
\end{equation}
The summation is over the set of positive integers $\{s,t \mid [s,t]=N\}$,  $[s,t]$ being the least common multiple of $s,t$. It also holds that
\begin{equation} \label{(8)}
S (N,T^{l})=\sum_{[l,t]=N l} S (t,T)
\end{equation}
\end{thm}

\begin{proof}
In order to prove  \eqref{(7)} it suffices to consider the equivalent formula (see \cite{car})
\begin{equation*}
\sum_{k|N}\sum_{[s,t]=k } S (s,T_{1}) S (t,T_{2})
= \sum_{k|N} S (k,T_{1} \otimes T_{2})
\end{equation*}
Using M\"obius inversion formula, the left hand side is equal to
\begin{equation*}
\sum_{s|N} S (s,T_{1}) \sum_{t|N} S (t,T_{2})
= (Tr T_{1}^{N}) (Tr T_{2}^{N})
\end{equation*}
But
$(Tr T_{1}^{N}) (Tr T_{2}^{N})=Tr (T_{1} \otimes T_{2})^{N}$. By
M\"obius inversion formula this gives the right hand side of the equivalent formula.
Using ideas from \cite{mor},
the next identity can be proved using the following equivalent formula:
\begin{equation*}
\sum_{g|N} \sum_{[l,t]=\frac{Nl}{g}} S (t,T)
= \sum_{g|N} S (N,T^{l})
\end{equation*}
The left hand side is equal to
$\sum_{t|lN} S(t,T)=Tr T^{lN}=Tr (T^{l})^{N}$.
Apply M\"obius inversion formula to get the result.
\end{proof}

\begin{rem} \label{rmk3}
Formula \eqref{(7)} may be generalized to the case $T=T_{1} \otimes T_{2} \otimes ...\otimes T_{l}$ to give
\begin{equation} \label{(9)}
\sum_{[s_{1}, \dots ,s_{l}]=N } S (s_{1},T_{1}) \dots S (s_{l},T_{l})
=  S (N,T)
\end{equation}
Also, it can be proved that
\begin{equation} \label{(13)}
S(N,T_{1}^{s} \otimes T_{2}^{r})=
\sum_{[rp,sq]=nrs} S (p,T_{1})S (q,T_{2})
\end{equation}
where $r$ and $s$ are relatively prime and the summation is over all positive integers $p,q$ such that $[rp,sq]=nrs$. The proof is an application of previous identities as in \cite{metro}, Theorem 5.
\end{rem}

\begin{rem} \label{rmk4}
In terms of $\Omega$, using that $[s,t](s,t)=st$, (7) becomes
\begin{equation} \label{(14)}
\sum_{[s,t]=N} (s,t) \Omega(s,T_{1}) \Omega (t,T_{2})
= \Omega (N,T_{1} \otimes T_{2})
\end{equation}
where $(s,t)$ is the maximum common divisor of $s$ and $t$.
This can be extended to the general case \eqref{(9)} to give
\begin{equation}
\sum_{[s_{1}, \dots ,s_{l}]=N } (s_{1}, \cdots, s_{l}  )\Omega (s_{1},T_{1}) \dots \Omega (s_{l},T_{l})
=  \Omega (N,T)
\end{equation}
where $(s_{1}, \cdots, s_{l})$  is the greatest common divisor of $(s_{1}, \cdots,s_{l})$ and the sum runs over all positive integers $(s_{1}, \cdots,s_{l})$ with least common multiple equal to $N$ and $T=T_{1} \otimes \dots \otimes T_{l}$. Also,
from (8),
\begin{equation} \label{(15)}
\Omega (N,T^{l})=\sum_{[l,t]=N l} \frac{t}{N}\Omega (t,T).
\end{equation}
In terms of $\Omega$ (10) reads
\begin{equation*}
N \Omega (N,T_{1}^{s} \otimes T_{2}^{r})=
\sum_{[rp,sq]=Nrs} pq \Omega (p,T_{1}) \Omega(q,T_{2})
\end{equation*}
Using $(rp,sq)[rp,sq]=rpsq$ with $[rp,sq]=Nrs$ implies $(rp,sq)N=pq$ and
\begin{equation*}
\Omega (N,T_{1}^{s} \otimes T_{2}^{r})=
\sum_{[rp,sq]=Nrs} (rp,sq) \Omega (p,T_{1}) \Omega(q,T_{2})
\end{equation*}
Replace $s$ and $r$ by $s/(r,s)$ and $r/(r,s)$ to get
\begin{equation}
(r,s) \Omega (N,T_{1}^{s/(r,s)} \otimes T_{2}^{r/(r,s)})=
 \sum (rp,sq) \Omega (p,T_{1}) \Omega(q,T_{2})
\end{equation}
The sum is over $p,q$  such that $pq/(pr,qs)=N/(r,s)$.
\end{rem}

Another identity satisfied by $\Omega$ is of the Witt type \eqref{(2)}. It is a well known result about the   $\zeta$ function of a graph $G$  which is defined as follows:
\begin{equation} \label{(17)}
\zeta(z):=\prod_{[p]} (1-z^{l(p)})^{-1}=\prod_{N=1}^{\infty} (1-z^{N})^{-\Omega(N,T)}
\end{equation}
See \cite{te} and  \cite{tee}. The first product is over the equivalence classes $[p]$ of backtrack-less 
and tail-less closed paths
 of length $l(p)$ in $G$.
It is a famous result that 
$\zeta=[det(1-zT)]^{-1}$, hence, 
$\Omega(N,T)$ satisfies the Witt type identity
\begin{equation} \label{(18)}
\prod_{N=1}^{\infty} (1-z^{N})^{\Omega(N,T)}=det(1-zT)
\end{equation}
As mentioned in the introduction, the coefficient of the linear term in Witt's identity is the negative of the number of loops in a graph with $R$ loops hooked to a single vertex.  The coefficients in the expansion of the determinant $det(1-zT)$ as a polynomial in $z$ also have  nice combinatorial meanings
related to the structure of the graph $G$ as proved by several authors. See \cite{storm} and  references therein. Next theorem gives formulas for these coefficients and for those of the inverse of the determinant which are relevant for the next section.

\begin{thm} \label{thm5} 
Define
\begin{equation} \label{(19)}
 g(z):=\sum_{N=1}^{\infty} \frac{ Tr T^{N} }{N} z^{N}.
 \end{equation}
Then,
\begin{equation} \label{(20)}
\prod_{N=1}^{+\infty}  (1-z^{N})^{\pm \Omega(N,T)}
 = 
e^{\mp g(z)}  =[det(1-zT)]^{\pm}= 1 \mp \sum_{i=1}^{+\infty} c_{\pm}(i) z^{i},\\
\end{equation}
where
\begin{equation} \label{(21)}
c_{\pm}(i)= \sum_{m=1}^{i} \lambda_{\pm}(m) \sum_{
\begin{array}{l} a_{1}+2a_{2}+...+ia_{i} =i\\
a_{1}+...+a_{i} = m \end{array}} 
 \prod_{k=1}^{i} 
\frac{(Tr T^{k})^{a_{k}}}{a_{k}! k^{a_{k}}}
\end{equation}
with $\lambda_{+}(m)=(-1)^{m+1}$, $\lambda_{-}(m)=+1$, $c_{+}(i)=0$ for $i > 2|E|$, and $c_{-}(i) \geq 0$. Furthemore,
\begin{equation} \label{(22)}
Tr T^{N} = N  \sum_{
\begin{array}{l}  s = (s_{i})_{i \geq 1}, s_{i} \in {\bf Z}_{\geq 0}\\
 \sum is_{i}=N  \end{array}} (\pm 1)^{|s|+1}
 \frac{(\mid s \mid -1)!}{s!} \prod c_{\pm}(i)^{s_{i}}\\
\end{equation}
where
$\mid s \mid = \sum s_{i}, s! = \prod s_{i} !$.
\end{thm}

\begin{proof}
Define $P_{\pm}$ by 
\begin{equation*}
P_{\pm}(z)=\prod_{N'=1}^{+\infty} (1-z^{N'})^{\pm \Omega(N',T)}
\end{equation*}
Take the logarithm of both sides and use \eqref{(4)} to get
\begin{eqnarray*}
ln P_{\pm} &=&\mp \sum_{N'} \sum_{k} \frac{1}{k} \Omega(N',T) z^{N' k}
=\mp \sum_{N=1}^{+\infty} \sum_{k|N}\frac{1}{k} \Omega\left(\frac{N}{k}, T \right) z^{N}\\
&=& \mp \sum_{N=1}^{+\infty} \frac{Tr T^{N}}{N} z^{N} = \mp g(z)
\end{eqnarray*}
from which the first equality in \eqref{(20)} follows. 
From  the definition of $g(z)$, it follows that
\begin{eqnarray*}
\mp g(z):= \mp \sum_{N=1}^{\infty} \frac{Tr T^{N}}{N} z^{N} &=& \mp Tr \sum_{N=1}^{+\infty} \frac{1}{N} T^{N}  z^{N}
= \pm Tr \hspace{1mm}  ln(1-zT)\\
&=& \pm ln \hspace{1mm}  det(1-zT)\\
\end{eqnarray*}
proving the second equality in \eqref{(20)}.

The third equality is obtained formally expanding the exponential.
As the formal Taylor
expansion of $1-e^{\mp g}$, the coefficients  $c_{\pm}$ are given by
\begin{equation*}
c_{\pm}(i) =
\frac{1}{i!}
\frac{d^{i}}{d z^{i}} \left[ \pm(1-e^{\mp g}) \right]|_{z=0}
\end{equation*}
Using Faa di Bruno's formula as in \cite{coss},  the derivatives  can be computed explicitly and \eqref{(21)} follows. The determinant is a polynomial of maximum degree $2|E|$, hence, $c_{+}(i)=0$ for $i>2|E|$. Clearly, $c_{-}(i) \geq 0$.

To prove \eqref{(22)} write \cite{kangg}
\begin{eqnarray*}
\mp ln \left( 1 \mp \sum_{i} c_{\pm}(i) z^{i} \right) 
&=& \mp \sum_{l=1}^{+\infty} \frac{(-1)}{l} \left( \pm \sum_{i} c_{\pm}(i) z^{i} \right)^{l}\\
 &=& \pm \sum_{l=1}^{+\infty} \frac{(\pm 1)}{l} 
\sum_{\begin{array}{l}  s = (s_{i})_{i \geq 1}\\
s_{i} \in {\Bbb Z}_{\geq 0}\\
 \sum s_{i}=l  \end{array}} 
 \frac{(\sum s_{i})!}{\prod (s_{i} !)} \left(\prod c_{\pm}(i)^{s_{i}} \right) z^{\sum s_{i} l}\\
 &=& \sum_{k=1}^{+\infty}    z^{k} 
 \sum_{
\begin{array}{l}  s = (s_{i})_{i \geq 1}, s_{i} \in {\Bbb Z}_{\geq 0}\\
 \sum is_{i}=k  \end{array}} 
 (\pm 1)^{|s|+1} \frac{(\mid s \mid -1)!}{s!} \prod c_{\pm}(i)^{s_{i}}    
\end{eqnarray*}

The second equality in \eqref{(20)} applied to the left hand side yields
\begin{equation*}
\mp ln \left( 1\mp \sum_{i} c_{\pm}(i) z^{i} \right) =\sum_{k=1}^{+\infty} \frac{Tr T^{k}}{k} z^{k}
\end{equation*}
Compare coefficients to get the result.
\end{proof}

\begin{rem} \label{rmk5}
 Witt identity can be expressed in terms of a determinant:
\begin{equation*}
\prod_{N=1}^{\infty} (1-z^{N})^{{\cal M}(N;R)} = 1- Rz=det(1-zQ)
\end{equation*}
The proof is analogous to the proof of previous theorem using \eqref{(5)}.
\end{rem}

\section{ $\Omega$ and free Lie super algebras} \label{sec:lie}

As mentioned in the introduction
the coefficient in the Witt formula \eqref{(1)} has an algebraic interpretation as the negative of the dimension of a vector space that generates a free Lie algebra and the inverse of this formula is the generating function of dimensions of the subspaces of the enveloping algebra of the Lie algebra. Is it possible that the coefficients in the determinant $det(1-zT_{G})$ above have similar interpretation? 	The answer is positive. Formula \eqref{(4)} and the coefficients of the determinant and its inverse can be interpreted 
algebraically in terms of data related to Lie super algebras.

In series of papers S. -J. Kang and M. -H Kim \cite{kangg, kangggg} generalized Witt formula \eqref{(1)} to the case the free Lie algebra $L$ is generated by an infinite graded vector space. They obtained a generalized Witt formula for the dimensions of the homogeneous subspaces of $L$ which satisfies a generalized Witt identity. In \cite{kangggg} S. -J. Kang extended these results to super spaces and Lie super algebras. Some of the results are summarized in the following proposition.

\begin{prop} \label{pr6}
 Let $V= \bigoplus_{i=1}^{\infty}
V_{i}$ be a ${\mathbb{Z}}_{>0}$-graded super space  with finite dimensions $dim V_{i}= |t(i)|$ and super dimensions
$Dim V_{i}= t(i) \in {\mathbb{Z}}$, $\forall i \geq 1$. 
Let ${\cal L}=
\bigoplus_{N=1}^{\infty} {\cal L}_{N}$ be the free Lie super algebra generated
by $V$ with a ${\mathbb{Z}}_{>0}$-gradation induced by that of $V$. Then, the
super dimensions of the subspaces ${\cal L}_{N}$ are given by formula
\begin{equation} \label{(23)}
Dim {\cal L}_{N}= \sum_{g | N} \frac{\mu(g)}{g}  W \left(\frac{N}{g}\right)
\end{equation}
The summation ranges over all common divisors $g$ of $N$ and $W$ is given by
\begin{equation} \label{(24)}
 W(N)= \sum_{s \in T(N)} \frac{(\mid s \mid -1)!}{s!} \prod t(i)^{s_{i}}
\end{equation}
where $T(N)=\{ s = (s_{i})_{i \geq 1} \mid s_{i} \in {\mathbb{Z}}_{\geq 0}, 
\sum is_{i}=N \}$
and $\mid s \mid = \sum s_{i}$, $ s! = \prod s_{i} !$.
The numbers $Dim {\cal L}_{N}$  satisfy the  identity
\begin{equation}\label{(25)}
\prod_{N=1}^{\infty} (1-z^{N})^{Dim {\cal L}_{N}}= 1- \sum_{i=1}^{\infty} t(i) z^{i}.
\end{equation}
The right hand side of \eqref{(25)} is  related to the generating function for the $W$'s,
\begin{equation}\label{4.7}
g(z) :=\sum_{n=1}^{\infty} W(n)z^{n}
\end{equation}
by the relation
\begin{equation}\label{4.8}
e^{-g(z)}=  1-\sum_{i=1}^{\infty} t(i) z^{i}
\end{equation}
Furthermore,
\begin{equation}\label{4.9}
\prod_{N=1}^{\infty} \left( \frac{1}{1-z^{N}} \right)^{Dim {\cal L}_{N}}= 
1+\sum_{i=1}^{\infty} Dim U({\cal L})_{i} z^{i}
\end{equation}
where $Dim U({\cal L})_{i}$ is the dimension of the $i$-th homogeneous subspace of $U({\cal L})$, the
 universal enveloping algebra of ${\cal L}$.
\end{prop}

In \cite{kangggg}, \eqref{(23)}  is called the {\it generalized Witt formula}; $W$ is
called {\it Witt partition function}; and \eqref{(25)} is called {\it
generalized Witt identity}.

See section 2.3 of \cite{kangggg}. Given a formal power series $\sum_{i=1}^{+\infty} t_{i} z^{i}$ with $ t_{i} \in {\Bbb Z}$, for all $i \geq 1$,
the coefficients in the series can be interpreted as the super dimensions of a ${\Bbb Z}_{>0}$-graded  super space $V= \bigoplus_{i=1}^{\infty}
V_{i}$ with dimensions $dim V_{i}= |t_{i}|$ and super dimensions
$Dim V_{i}= t_{i} \in {\Bbb Z}$.
Let ${\cal L}$ be the free Lie super algebra generated by $V$. Then, it has a gradation induced by $V$ and the homogeneous subspaces have dimension given by \eqref{(23)}.  Apply this interpretation to the  determinant $det(1-zT)$ which is a polynomial of degree $2|E|$ in the formal variable $z$. It can be taken as a power series with coefficients $t_{i}=0$, for $i > 2|E|$.
Comparison of the formulas
in Theorem \ref{thm5} with formulas
in the above Proposition  implies the next result:

\begin{prop} \label{7}
Given a graph $G$, $T$ its edge matrix, 
let
$V= \bigoplus_{i=1}^{2|E|}
V_{i}$ be a ${\mathbb{Z}}_{>0}$-graded super space with finite dimensions 
$dim V_{i}= |c_{+}(i)|$ and the super dimensions
$Dim V_{i}= c_{+}(i)$ given by \eqref{(21)}, the coefficients of $det(1-zT)$. Let ${\cal L}=
\bigoplus_{N=1}^{\infty} {\cal L}_{N}$ be the free Lie super algebra generated
by $V$. Then, the
super dimensions of the subspaces ${\cal L}_{N}$ are given by $Dim{\cal L}_{N} =\Omega(N, T)$. 
The algebra has generalized Witt identity given by \eqref{(18)}.
The zeta function of $G$ \eqref{(17)} is the generating function for the dimensions of the subspaces of
the enveloping algebra $U({\cal L})$ of ${\cal L}$ which are given by
$Dim U({\cal L})_{n}=c_{-}(n)$, $c_{-}(n)$ given by \eqref{(21)}
.
\end{prop}

\begin{exmp}
$G_{1}$, the graph with $R \geq 2$ edges counterclockwisely oriented and hooked to a single vertex shown in Figure 1. The edge matrix for $G_{1}$ is the $2R \times 2R$ symmetric matrix
\begin{equation*}
T_{G_{1}} = \left( \begin{array}{clcr}
A & B\\
B & A 
\end{array} \right)
\end{equation*}
\begin{center}
\begin{figure}[h]
\centering
\includegraphics[scale=0.5]{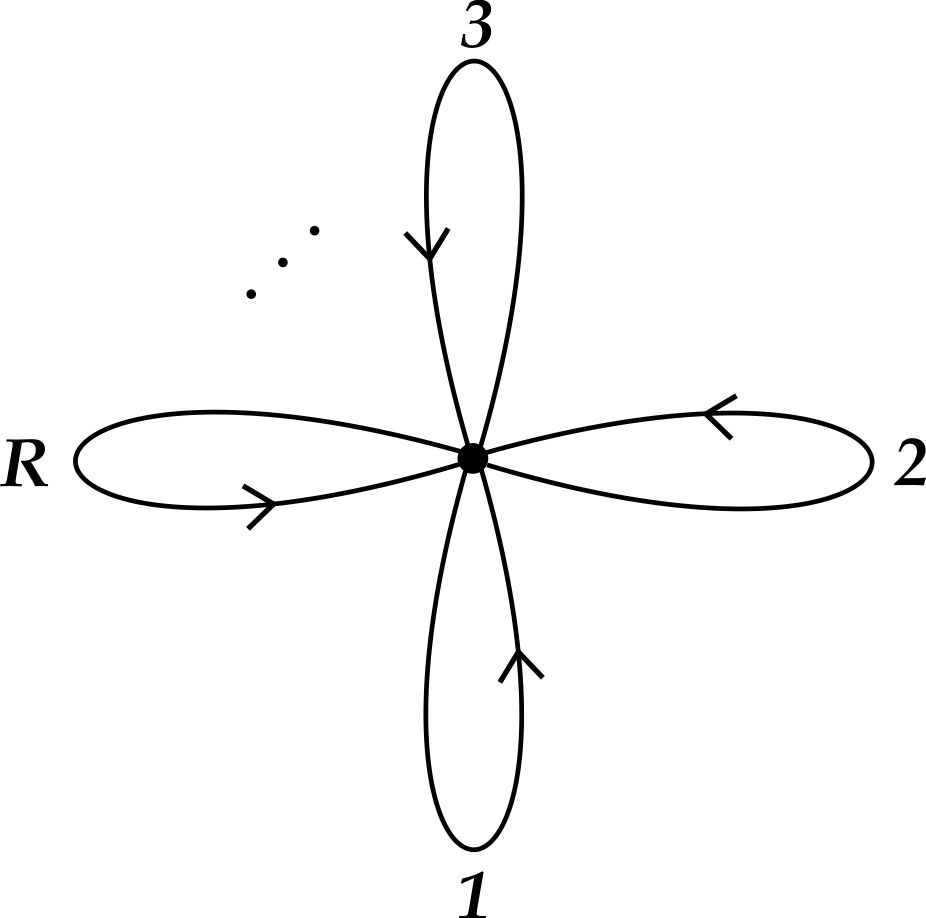}
\caption{Graph $G_{1}$}
\label{Fi:G1}
\end{figure}
\end{center}
where $A$ is the $R \times R$ matrix with all entries equal to $1$ and $B$  is the $R \times R$ matrix with  the main diagonal entries  equal to $0$ and all the other entries equal to $1$. This matrix has the trace given by
\begin{equation*}
Tr T_{G_{1}}^{N} = 1+(R-1)(1+(-1)^{N})+(2R-1)^{N}, \hspace{2mm} N=1,2, \dots
\end{equation*}
and the determinant
\begin{eqnarray*}
det(1-zT_{G_{1}}) &=& (1-z) \left[ 1-(2R-1)z \right](1-z^{2})^{R-1}\\
            &=& 1-\sum_{i=1}^{2R} c(i) z^{i}
\end{eqnarray*}
where $c(2R)=(-1)^{R}(2R-1)$,
\begin{equation*}
c(2i)=(-1)^{i}(2i-1) 
\left( \begin{array}{c}
R\\
i
\end{array} \right), \hspace{2mm} i=1, \cdots, R-1
\end{equation*}
and
\begin{equation*}
c(2i+1)=2R(-1)^{i} \left( \begin{array}{c}
R-1\\
i
\end{array} \right), \hspace{2mm} i=0,1, \cdots R-1
\end{equation*}
Futhermore,
\begin{eqnarray*}
[det(1-zT_{G_{1}})]^{-1} &=& \sum_{q=0}^{+\infty} z^{q} \sum_{i=0}^{q} a_{i} (2R-1)^{q-i}
\end{eqnarray*}
where
\begin{equation*}
a_{i}= \sum_{k=0}^{i} (-1)^{i-k} 
\left( \begin{array}{c}
k+R-1\\
R-1
\end{array} \right)
\left( \begin{array}{c}
i-k+R-2\\
R-2
\end{array} \right)
\end{equation*}

Let's consider the case $R=2$. In this case,
\begin{equation*}
Tr T_{G_{1}}^{N}=2+(-1)^{N}+3^N, \hspace{5mm} det(1-zT_{G_{1}})= 1-4z+2z^2+4z^3-3z^4
\end{equation*}
so that the number of classes of reduced nonperiodic cycles of length $N$ is given by the formula
\begin{equation*}
\Omega(N, T_{G_{1}}) = \frac{1}{N} \sum_{g|N} \mu (g) \left( 2+(-1)^{\frac{N}{g}}+3^{\frac{N}{g}} \right)
\end{equation*}
The graph generates the following algebra.
Let
$V= \bigoplus_{i=1}^{4}
V_{i}$ be a ${\bf Z}_{>0}$-graded super space  with dimensions  $dimV_{1}=4$, $dimV_{2}=2$, $dimV_{3}=4$, $dimV_{4}=3$ and super dimensions $DimV_{1}=-4$, $DimV_{2}=2$, $DimV_{3}=4$, $DimV_{4}=-3$.  Let ${\cal L}=
\bigoplus_{N=1}^{\infty} {\cal L}_{N}$ be the free graded Lie super algebra generated
by $V$.  The
dimensions of the subspaces ${\cal L}_{N}$ are given by the generalized Witt formula
\begin{equation*}
Dim{\cal L}_{N} =\frac{1}{N} \sum_{g|N} \mu (g) \left( 2+(-1)^{\frac{N}{g}}+3^{\frac{N}{g}} \right)
\end{equation*}
which satisfies the generalized Witt identity
\begin{eqnarray*}
\prod_{N=1}^{+\infty} (1-z^{N})^{\Omega(N, T_{G_{1}})}
&=&  1-[4z-2z^2-4z^3+3z^4]\\
\end{eqnarray*}
 The subspace $U_{n}({\cal L})$ of the enveloping algebra $U({\cal L})$ have dimensions given by 
the zeta function of the graph,
\begin{eqnarray*}
\prod_{N=1}^{+\infty} (1-z^{N})^{-\Omega(N, T_{G_{1}})}
&=& 1+\frac{1}{16}\sum_{n=1}^{\infty} ((-1)^{n}+39 \cdot 3^{n}-24-12n) z^{n}\\
\end{eqnarray*}
so that
\begin{displaymath}
Dim U_{n}({\cal L})=\frac{1}{16} \left(  (-1)^{n}+39 \cdot 3^{n}-24-12n\right)
\end{displaymath}
\end{exmp}

\begin{exmp}
$G_{2}$, the bipartite graph shown in Figure 2.
\begin{center}
\begin{figure}[h]
\centering
\includegraphics[scale=0.5]{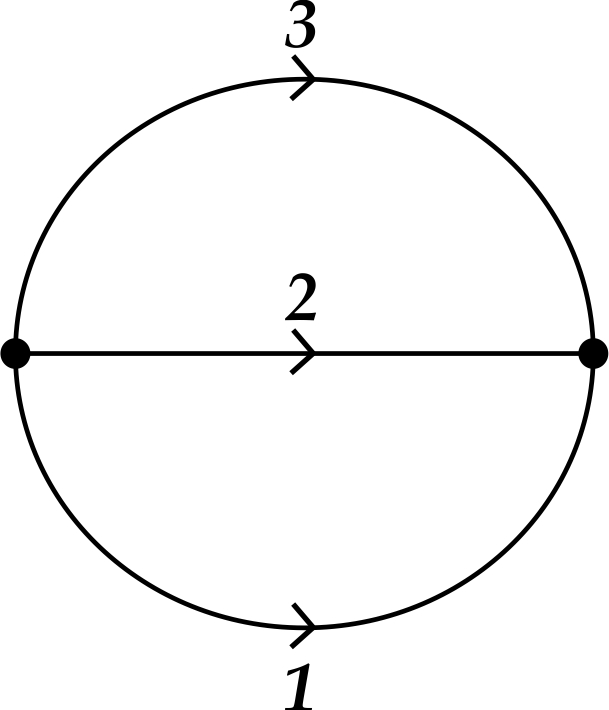}
\caption{Graph $G_{2}$}
\label{Fi:G2}
\end{figure}
\end{center}
The edge matrix of $G_{2}$ is
\begin{equation*}
T_{G_{2}} = \left( \begin{array}{clcrclcr}
0 & 0 & 0 & 0 & 1 & 1\\
0 & 0 & 0 & 1 & 0 & 1 \\
0 & 0 & 0 & 1 & 1 & 0 \\
0 & 1 & 1 & 0 & 0 & 0\\
1 & 0 & 1 & 0 & 0 & 0\\
1 & 1 & 0 & 0 & 0 & 0 
\end{array} \right)
\end{equation*}
\\
\\
The matrix has the trace
$Tr T_{G_{2}}^{N}= 0$ if $N$ is odd and $Tr T_{G_{2}}^{N}=
4+2 \cdot 2^{N}$ if $N$ is even, and the determinant
\begin{equation*}
det(1-zT_{G_{2}})= 1-6z^2+9z^4-4z^6
\end{equation*}
If $N$ is odd, the number of classes of nonperiodic cycles of length $N$ is $\Omega(N, T_{G_{2}} )=0$, if $N$ is odd, and
\begin{equation*}
\Omega (N, T_{G_{2}})= \frac{1}{N}\sum_{g|N} \mu (g) Tr T_{G_{2}}^{\frac{N}{g}} 
\end{equation*}
if $N$ is even. The graph generates the following algebra.
Let
$V= \bigoplus_{i=1}^{3}
V_{2i}$ be a ${\bf Z}_{>0}$-graded superspace  with dimensions  $dimV_{2}=6$, $dimV_{4}=9$, $dimV_{6}=4$ and superdimensions $DimV_{2}=6$, $DimV_{4}=-9$, $DimV_{6}=4$.  Let ${\cal L}=
\bigoplus_{N=1}^{\infty} {\cal L}_{N}$ be the free graded Lie superalgebra generated
by $V$.  The
dimensions of the subspaces ${\cal L}_{N}$ are $Dim{\cal L}_{N}=0$, for $N$ odd and
\begin{equation*}
Dim{\cal L}_{N} =\frac{1}{N}\sum_{g|N} \mu (g) Tr T_{G_{2}}^{\frac{N}{g}} 
\end{equation*}
for $N$ even. The dimensions satisfy 
the generalized Witt identity
\begin{eqnarray*}
\prod_{N=1}^{+\infty} (1-z^{N})^{\Omega(N, T_{G_{2}})}
&=&  1-[6z^2-9z^4+4z^6]
\end{eqnarray*}
The generating function for the dimensions of the subspaces $U_{n}({\cal L})$ of the enveloping algebra $U({\cal L})$ is given by 
\begin{eqnarray*}
\prod_{N=1}^{+\infty} (1-z^{N})^{-\Omega(N, T_{G_{2}})}
&=& 1+\frac{1}{18}\sum_{n=1}^{\infty} (2^{2n+5}-6n-14) z^{2n}\\
\end{eqnarray*}
so that
\begin{equation*}
Dim U_{n}({\cal L})= 2^{2n+5}-6n-14
\end{equation*}
\end{exmp}

\section{Necklace colorings induced by paths} \label{sec:col}

Given the set of $2|E|$ colors $\{c_{1},..., c_{2|E|} \}$,
assign the colors $c_i, c_{|E|+i}$ to edges $e_{i}, e_{|E|+i}=e_{i}^{-1} \in G^*$, respectively, so that
to a cycle of length $N$ in $G$ corresponds an ordered sequence of $N$ colors. Now, assign each color in this sequence to a bead in a circular string with $N$ beads - a necklace - in such a manner that two adjacent colors in the sequence are assigned to adjacent beads. The non backtracking condition for cycles implies that 
no two adjacent  beads are painted with colors, say, $c_i$ and $c_{|E|+i}$.
It is clear that there is a 
 correspondence between the classes of nonperiodic cycles of length $N$ in $G$ and classes of  nonperiodic colorings of a necklace with $N$ beads with at most $2|E|$ distinct colors  induced by the cycles so that the number of inequivalent colorings is $\Omega_{G}(N, T)$.  Of course
the structure of the graph reflects itself  on the coloring.
For instance, the presence of loops in the graph means that their assigned colors  may appear repeated in a string of adjacent beads.  This can not happen to a color assigned to an edge which is not a loop. 
 The edge matrix $T$ may be called the {\it color matrix}. It basically tells what colors are allowed to follow a given color in the necklace. 
Element $T_{ij}=1$, if a color $c_j$ can follow color $c_i$ and $c_j \neq
c_{|E|+i}$;
$T_{ij}=0$, otherwise. 

\begin{exmp}
 The number of nonperiodic colorings of a necklace with $N$ beads with at most 6 colors and color matrix given by graph $G_{1}$ with $R=3$:
\begin{equation*}
T_{G_{1}} = \left( \begin{array}{clcrclcr}
1 & 1 & 1 & 0 & 1 & 1\\
1 & 1 & 1 & 1 & 0 & 1 \\
1 & 1 & 1 & 1 & 1 & 0 \\
0 & 1 & 1 & 1 & 1 & 1 \\
1 & 0 & 1 & 1 & 1 & 1 \\
1 &  1 & 0 & 1 & 1 & 1\\
\end{array} \right)
\end{equation*}
is
\begin{equation*}
\Omega(N,T_{G_{1}}) = 
\frac{1}{N} \sum_{g|N} \mu (g) \left( 3+2(-1)^{\frac{N}{g}}+5^{\frac{N}{g}} \right)
\end{equation*}
For $N=3$, $\Omega(3, T_{G_{1}})=40$. The classes of nonperiodic
colorings are $[c_{i}c_{j}^{2}]$,
$[c_{3+i} c_{j}^{2}]$, $[c_{i} c_{3+j}^{2}]$,
$[{c_{3+i}} \hspace{1mm} {c_{3+j}}^{2}]$,
$[c_{i}^{2} c_{j}]$, $[{c_{3+i}}^{2} c_{j}]$,
$[c_{i}^{2} {c_{3+j}}]$, $[{c_{3+i}}^{2} {c_{3+j}}]$, for $(i,j) = (1,2), (1,3), (2,3)$ and
$[c_{i}c_{j}c_{k}]$, $[{c_{3+i}} c_{j}c_{k}]$, $[c_{i} {c_{3+j}} c_{k}]$,$[c_{i} c_{j}{c_{3+k}}]$, $[{c_{3+i}} \hspace{1mm} { c_{3+j}} c_{k}]$, $[{c_{3+i}} c_{j}{c_{3+k}}]$, $[c_{i} {c_{3+j}} \hspace{1mm} {c_{3+k}}]$, $[{c_{3+i}} \hspace{1mm} {c_{3+j}} \hspace{1mm} {c_{3+k}}]$, for $(i,j,k) = (1,2,3), (1,3,2)$. These corresponds to the classes of cycles
$[e_{i}^{+1} e_{j}^{+2}]$,
$[e_{i}^{-1} e_{j}^{2}]$, $[e_{i}^{+1} e_{j}^{-2}]$,
$[e_{i}^{-1} e_{j}^{-2}]$, $[e_{i}^{+2} e_{j}^{+1}]$,
$[e_{i}^{-2} e_{j}^{+1}]$, $[e_{i}^{+2} e_{j}^{-1}]$, $[e_{i}^{-2} e_{j}^{-1}]$, for $(i,j) = (1,2), (1,3), (2,3)$ and
$[e_{i}^{+1} e_{j}^{+1}e_{k}^{+1}]$, $[e_{i}^{-1} e_{j}^{+1}e_{k}^{+1}]$, $[e_{i}^{+1} e_{j}^{-1}e_{k}^{+1}]$,$[e_{i}^{+1} e_{j}^{+1}e_{k}^{-1}]$, $[e_{i}^{-1} e_{j}^{-1}e_{k}^{+1}]$, $[e_{i}^{-1} e_{j}^{+1}e_{k}^{-1}]$,$[e_{i}^{+1} e_{j}^{-1}e_{k}^{-1}]$, $[e_{i}^{-1} e_{j}^{-1}e_{k}^{-1}]$, for $(i,j,k) = (1,2,3), (1,3,2)$.
\end{exmp}

\begin{exmp}
For the graph $G_2$ assign to the three oriented edges $e_{1},e_{2},e_{3}$ the colors $c_1, c_2, c_3$, respectively. 
The graph is bipartite so only paths of even length are possible.   The number of inequivalent nonperiodic 
colorings induced by the nonperiodic cycles of length $N$ is given by 
\begin{equation}
\Omega(N, T_{G_{2}})= \frac{1}{N}
\sum_{g|N} \mu (g)
Tr T_{G_{2}}^{\frac{N}{g}}
\end{equation}
For $N=2$, $\Omega(2, T_{G_{2}})=6$. The classes are $[e_{1} e_{2}^{-1}]$, $[e_{1}^{-1} e_{2}]$, $[e_{1} e_{3}^{-1}]$, $[e_{1}^{-1} e_{3}]$, $[e_{2} e_{3}^{-1}]$,  $[e_{2}^{-1} e_{3}]$. The induced colorings are $[c_{1} {c_{5}}]$, $[{c_{4}}, c_{2}]$, $[c_{1} {c_{6}}]$, $[{c_{4}} c_{3}]$, $[c_{2} {c_{6}}]$,  $[{c_{5}} c_{3}]$.
For $N=4$ , $\Omega(4, T_{G_{2}})=6$. The paths are  $[e_{1} e_{2}^{-1} e_{3} e_{2}^{-1}]$,
$[e_{1} e_{2}^{-1} e_{1} e_{3}^{-1}]$, $[e_{1} e_{3}^{-1} e_{2} e_{3}^{-1}]$, $[e_{1}^{-1} e_{2} e_{1}^{-1}e_{3}]$,
 $[e_{1}^{-1} e_{2} e_{3}^{-1}e_{2}]$, $[e_{1}^{-1} e_{3} e_{2}^{-1}e_{3}]$.
To these classes correspond the colorings
 $[c_{1} {c_{5}} c_{3} {c_{5}}]$,
$[c_{1} {c_{5}} c_{1} {c_{6}}]$, $[c_{1} {c_{6}} c_{2} {c_{6}}]$, $[{c_{4}} c_{2} {c_{4}}c_{3}]$,
 $[{c_{4}} c_{2} {c_{6}}c_{2}]$, $[{c_{4}} c_{3} {c_{5}}c_{3}]$.
The graph has no loops so strings of two or more beads with a same color is not possible.
\end{exmp}

\begin{ack}
I would like to thank Prof. C. Storm (Adelphi University, USA) for sending
me his joint paper with G. Scott on the coefficients of Ihara zeta function. Also, many thanks to Prof.
Asteroide Santana for his help with the figures, latex commands and determinants.
\end{ack}

\end{document}